\documentclass[english]{amsart}

\usepackage{latexsym}

\usepackage{pstricks,pst-node,pst-tree}

\usepackage{amsmath}
\usepackage{rotating}
\usepackage{amsfonts}
\usepackage{amsthm}
\usepackage{amssymb}
\usepackage{verbatim}
\usepackage{psfrag}
\usepackage{graphicx}
\newtheorem{lemm}{Lemma}[section]
\newtheorem{prop}[lemm]{Proposition}

\newtheorem{defi}[lemm]{Definition}
\newtheorem{coro}[lemm]{Corollary}
\newtheorem{rema}[lemm]{Remark}

\newtheorem{exem}[lemm]{Example}
\newtheorem{ques}[lemm]{Question}
\newtheorem{theo}{Theorem}

  \usepackage[all]{xy}
\xyoption{matrix}

\newcommand{\proofend}{\hfill $\square$}

\newcommand{\PGL}{\mathrm{PGL}}

\newcommand{\C}{\mathbb{C}}
\newcommand{\Bir}{\mathrm{Bir}}

\newcommand{\h}[1]{\hspace{-#1mm}}

\newcommand{\Aut}{\mathrm{Aut}}

\newcommand{\rkPic}[1]{\mathrm{rk\ Pic}(#1)}

\newcommand{\Pn}{\mathbb{P}^2}

\newcommand{\Pic}[1]{\mathrm{Pic}(#1)}

\newcommand{\p}{\mathbb{P}}

\title{Dynamical degrees of (pseudo)-automorphisms fixing cubic hypersurfaces}
\thanks{Supported by the SNSF grant no PP00P2\_128422 /1}
\author{J\'er\'emy Blanc}
\email{Jeremy.Blanc@unibas.ch}
\begin{document}
\begin{abstract}
We give a way to construct group of pseudo-automorphisms of rational varieties of any dimension that fix pointwise the image of a cubic hypersurface of $\mathbb{P}^n$. These group are free products of involutions, and most of their elements have dynamical degree $>1$. Moreover, the Picard group of the varieties obtained is not big, if the dimension is at least $3$.

We also answer a question of E. Bedford on the existence of birational maps of the plane that cannot be lifted to automorphisms of dynamical degree $>1$, even if we compose them with an automorphism of the plane.
\end{abstract}


\maketitle
\section{Introduction}
A birational map $\varphi\colon \p^n\dasharrow \p^n$ (or a Cremona transformation) is a rational map given by 
$$(x_0:\dots:x_n)\dasharrow (P_0(x_0,\dots,x_n):\dots:P_n(x_0,\dots,x_n)),$$
where all $P_i$ are homogeneous polynomials of the same degree, which admits an inverse of the same type. Choosing all $P_i$ without common component, the degree $\deg(\varphi)$ of $\varphi$ is by definition the degree of the polynomials $P_i$, or equivalently the degree of the pull-back of hyperplanes of $\p^n$ by $\varphi$.

The \emph{$($first$)$ dynamical degree} of $\varphi$ is the number
$$\lim_{n\to \infty} (\deg(\varphi^n))^{1/n},$$
which always exists, since $\deg (\varphi^{a+b})\le \deg(\varphi^a)\cdot \deg(\varphi^b)$ for any $a,b\ge 0$. It is moreover invariant under conjugation. 

There is a sequence of articles which provide families of examples of birational maps of $\p^2$ with dynamical degree $>1$, lifting to automorphisms of a smooth rational surface obtained by blowing-up a finite number of points (among them, see \cite{bib:BedKim06}, \cite{bib:BedDil06}, \cite{bibi:McMullen}, \cite{bib:BedKim09}, \cite{bib:BedKim10}, \cite{bib:Diller}, \cite{bib:DeGr}....) The general way of producing examples is to start by a simple birational map (quadratic involution, automorphism of the affine plane,...), and to compose it with a linear automorphism of $\p^2$ to impose that the base-points of the inverse "come back" to the base-points of the map after a certain number of iteration of the map.

This approach was generalised in dimension $3$, in \cite{bib:BedKimDim3}, to provide pseudo-automorphisms of projective $3$-folds of dynamical degree $>1$, starting from a special family of quadratics elements of $\Bir(\p^3)$. Recall that a pseudo-automorphism of $X$ is a birational self-map $\varphi\in\Bir(X)$ such that $\varphi$ and $\varphi^{-1}$ do not contract any codimension $1$ set; it is the same as automorphism for smooth projective surfaces.

Other examples of pseudo-automorphisms of dynamical degree $>1$ of rational projective varieties of any dimension were given in \cite{bib:PerZha}, using actions of Weyl groups on blow-ups of $\p^n$ at a finite number of points.

In this article, we give another way of constructing examples, which also works in any dimension. This produces large groups of pseudo-automorphisms, where almost all elements have dynamical degree $>1$. Moreover, the rank of the Picard group of the varieties obtained does not need to be very large, contrary to what happens in dimension $2$, or to the examples of \cite{bib:BedKimDim3} and \cite{bib:PerZha}.\\

We recall the following construction, defined in \cite[page~42, Example~3]{bib:Giz} over the name of $R_p$.
\begin{defi} \label{Def:InvPn} 
Let $Q\subset \p^n$ be a cubic hypersurface, and let $p\in Q$ be a smooth point. We define an involution $\sigma_{p,Q}\in \Bir(\p^n)$ which fixes pointwise $Q$ by the following: if $L$ is a general line of $\p^n$ passing through $p$, we have $\sigma_{p,Q}(L)=L$ and the restriction of $\sigma_{p,Q}$ to $L$ is the involution that fixes $(L\cap Q)\backslash\{p\}$.
\end{defi}
From the geometric definition, we can easily get an algebraic definition by polynomials (see \cite{bib:Giz} or Section~\ref{Sec:SigmaP}). We will show that any $\sigma_{p,Q}$ lifts to an automorphism of a smooth variety obtained by blowing-up codimension $2$ subsets of $\p^n$. Taking a finite number of points on the same cubic gives a huge group of pseudo-automorphisms of a rational $n$-fold:

\begin{theo}\label{Thm:CubicsPseudo}
Let $Q\subset \p^n$ be a cubic hypersurface, and let $p_1,\dots,p_k\in Q$ be distinct smooth points. 
For $i=1,\dots,k$, we write $\Gamma_i=\{x\in Q \ |\ $the line through $x$ and $q_i$ is tangent to $Q$ at $x\}$.

Denote by $\pi\colon X\to \p^n$ the following birational morphism: it first blows-up  all points $p_1,\dots,p_k$, then blows-up the strict transform of $\Gamma_{1}$, then the strict transform of $\Gamma_{2}$ and so on until blowing-up the strict transform of $\Gamma_{k}$. 

If $p_i\notin \Gamma_{j}$ for any $i\not= j$, then $\sigma_{p_1,Q},\dots,\sigma_{p_k,Q}$ lift to pseudo-automorphisms $\hat{\sigma}_1,\dots,\hat{\sigma}_k$ of $X$, that generate a free product  $$G=\star_{i=1}^k <\hat{\sigma}_{i}>,$$ having the following properties:

\begin{enumerate}
\item
Any element of $G$ of finite order is conjugate to $\hat{\sigma}_{i}$ for some $i$ $($and has dynamical degree~$1)$;
\item
Any element conjugate to $(\hat{\sigma}_{i}\hat{\sigma}_{j})^m$ for $i\not=j$ and $m\ge 1$ is of infinite order, and its dynamical degree is equal to $1$ if and only if $d=3$;
\item
Any other element has dynamical degree $>1$.
\item
Each element of $G$ fixes pointwise the lift of the cubic $Q$ on $X$.
\end{enumerate}

\end{theo}

\begin{coro}\label{Coro1}
For any $n\ge 3$, there exist a rational smooth $n$-fold $X$ with $\rkPic{X}=7$ admitting a group of pseudo-automorphisms $G$ isomorphic to the free group with $2$ generators, such that all elements of $G\backslash \{1\}$ have dynamical degree $>1$. Moreover, $G$ fixes pointwise an hypersurface isomorphic to a general smooth cubic of $\p^n$.
\end{coro}
\begin{rema}Any pseudo-automorphism of a smooth projective surface $X$ with $\rkPic{X}\le 10$ has dynamical degree $1$. All previously known examples of smooth rational varieties admitting pseudo-automorphisms with dynamical degree $>1$ had Picard rank bigger than $10$.\end{rema}

Section~\ref{Sec:SigmaP} is devoted to the proof of Theorem~\ref{Thm:CubicsPseudo} and of its corollary. 

\begin{ques}
What is the minimal rank of a smooth projective rational $3$-fold admitting an automorphism of dynamical degree $>1$?
\end{ques}

\bigskip

Restricting to dimension $n=2$, Theorem~\ref{Thm:CubicsPseudo} gives the existence of group of automorphisms of rational surfaces with many elements of dynamical degree $>1$. The rank of the Picard group is however quite large, at least $16$. This is because the varieties $\Gamma_i$ are in fact union of $4$ distinct points. 

As we said above, the usual way to construct automorphisms of projective rational surfaces with dynamical degree $>1$ is to take a birational map of small degree, and then to compose it with an automorphism so that all base-points of the inverse  are sent after some iterations onto base-points of the map. This approach gives rise to the following question of Eric Bedford, also stated and studied by Julie D\'eserti and Julien Grivaux in \cite{bib:DeGr}:
\begin{ques}\label{QuesBed}
Does there exist a birational map of the projective plane $\varphi$ of degree $>1$ such that for all $\tau\in \Aut(\p^2)$ the map $\tau\varphi$ is not birationally conjugate to an automorphism of dynamical degree $>1$?
\end{ques}
Here, by "conjugate to an automorphism", we mean the existence of a birational map $\nu\colon \p^2\dasharrow X$, where $X$ is a projective smooth surface, such that $\nu (\tau\varphi)\nu^{-1}\in \Aut(X)$. 

It is know that if $\varphi$ is a birational map of degree $2$, there exists an automorphism $\tau$ such that $\tau\varphi$ is conjugate to an automorphism of dynamical degree $>1$. We recall this fact in Section~\ref{Sec:Dim2}. The same was also proved by different authors for some special maps $\varphi$ of degree $3$.
Using the involutions $\sigma_{p,Q}$, we prove in Section~\ref{Sec:Dim2} that the same holds for a \emph{general} map of degree $3$. 

The possible map $\varphi$ of Question~\ref{QuesBed} cannot thus be a general cubic transformation, or a transformation of degree $2$. Section~\ref{Sec:Example6} is devoted to the proof of the following result, showing that the map can be of degree $6$, and answering the question of Bedford, D\'eserti and Grivaux:
\begin{theo}\label{Thm:Exa6}
Let $\chi\colon\p^2\dasharrow \p^2$ be the birational map given by
$$\chi\colon (x:y:z)\dasharrow (xz^5+(yz^2+x^3)^2:yz^5+x^3z^3:z^6).$$

 For any automorphism $\tau\in \Aut(\p^2)$, the birational map $\tau\chi\in \Bir(\p^2)$ is not conjugate to an automorphism of a smooth projective rational surface.
\end{theo}
\begin{ques}
Does there exists a transformation of degree $<6$ having the above property?
\end{ques}

{\it Aknowledgements.}
The author thanks Julie D\'eserti for interesting discussions on the topic during the preparation of the article.
\section{The maps $\sigma_{p,Q}$, their lifts and the groups generated by these}\label{Sec:SigmaP}
Let us describe algebraically the map $\sigma_{p,Q}$ introduced in Definition~\ref{Def:InvPn}.

For this, choose homogeneous coordinates $(x_1:x_2:\dots:x_n:y)$ on $\p^n$ and assume, up to a change of coordinates, that $p$ is equal to $(0:\dots:0:1)$. The equation of $Q$ is thus $y^2P_{1}+yP_{2}+P_3$, where $P_{1},P_{2},P_3\in \C[x_1,\dots,x_n]$ are homogeneous of degree $1,2,3$. The involution $\sigma_{p,Q}$ sends a point $(x_1:x_2:\dots:x_n:y)$ onto 
$$(-x_1(P_{2}+2yP_{1}):\dots:-x_n(P_{2}+2yP_{1}):P_{2}y+2P_3).$$

The point $p$ is a base-point of multiplicity $2$. The subscheme $\Gamma_p\subset Q\subset \p^n$ of codimension $2$ given by $P_{2}+2yP_{1}=0$ and $yP_{2}+2P_3=0$ is also contained in the base-locus, and $\sigma_{p,Q}$ is defined on $\p^n\backslash \Gamma_p$.

 The cone $V_p\subset \p^n$ given by $(P_{2})^2-4P_3P_{1}$ is contracted onto $\Gamma_p$ by $\sigma_{p,Q}$, and the hypersurface given by $P_{2}+2yP_{1}$ is contracted onto the point $p$.

Note that $\Gamma_p$ is also given by the intersection of $Q$ with the hypersurface of equation $P_{2}+2yP_{1}=0$, or with the cone $V_p$, and corresponds to the points $q\in Q$ such that the line passing through $p$ and $q$ is tangent to $Q$ at $q$, as defined in the introduction.

\begin{prop}\label{Prop:Lifts}
Denote by $\pi_p\colon X_p\to \p^n$ the blow-up of $p$, by $\pi_\Gamma\colon X\to X_p$ the blow-up of the strict transform of $\Gamma_p$ and write $\pi=\pi_p\circ \pi_\Gamma\colon X\to \p ^n$. The lift $\hat{\sigma}_p=\pi^{-1}\sigma\pi$ of $\sigma_{p,Q}$ is an automorphism of $X$.

Writing $H\subset \Pic{X}$ the pull-back of an hyperplane of $\p^n$ by $\pi$, $E\subset \Pic{X}$ the pull-back of $\pi_{p}^{-1}(p)$ by $\pi_\Gamma$ and $F$ the exceptional divisor of $\pi_\Gamma$,  $(H,E,F)$ is a basis of a sub-$\mathbb{Z}$-module of $\Pic{X}$ invariant by $\hat{\sigma}_p$ $($this sub-module is equal to $\Pic{X}$ if and only if $\Gamma_p$ is irreducible$)$. Moreover, the action of $\hat{\sigma}_p$ relative to this basis is
$$\left[\begin{array}{rrr}
3 & 2& 4 \\
-2& -1&-4\\
-1& -1&-1
\end{array}\right].$$
\end{prop}
\begin{proof}
We can view $X_p$ in $\p^n\times \p^{n-1}$ as 
$$X_p=\left\{\big((x_1:x_2:\dots:x_n:y),(z_1:\dots:z_n)\big)\ \Big|\ x_iz_j=x_jz_i\mbox { for }1\le i,j\le n\right\},$$ where $\pi_p\colon X_p\to \p^n$ is given by the projection on the first factor. The variety $X_p$ is covered by open subsets $U_1,\dots,U_n$, where $U_i$ is the set where $z_i\not=0$.

Each $U_i$ is isomorphic to $\mathbb{A}^{n-1}\times \p^1$. For $i=1$, the isomorphism is given by:
$$\begin{array}{rcl}
\mathbb{A}^{n-1}\times \p^1&\stackrel{\simeq}{\longrightarrow} & U_1\\
\big((t_2,\dots,t_{n}),(\alpha:\beta)\big)&\longmapsto &\big((\alpha:\alpha t_2:\dots:\alpha t_{n}:\beta),(1:t_2:\dots:t_{n})\big).\\
\end{array}$$

The lift of $\sigma$ preserves $U_1$ and restricts to the following birational map
$$\big((t_1,\dots,t_n),(\alpha:\beta)\big)\dasharrow \big((t_1,\dots,t_n),(-(\alpha R_{2}+2\beta R_{1}):\beta R_{2}+2\alpha R_3)\big),$$
where $R_{i}=P_i(1,t_2,\dots,t_n)$. On this chart, $\Gamma_p$ is given by 
 $\alpha R_{2}+2\beta R_{1}=0$ and $R_{2}\beta +2\alpha R_3=0$. Blowing-up the corresponding ideal, we obtain the variety $W_1\subset \mathbb{A}^n\times \p^1 \times \p^1$ given by
 $$W_1=\big\{((t_1,\dots,t_n),(\alpha:\beta),(u:v))\ |\  u(R_{2}\beta +2\alpha R_3)=-v(\alpha R_{2}+2\beta R_1)\big\},$$
 and the blow-up is given by the projection $W_1\to \mathbb{A}^n\times \p^1$ on the first two factors.
 The map $\hat{\sigma}$ corresponds in these coordinates to 
 $$\left((t_1,\dots,t_n),(\alpha:\beta),(u:v)\right)\mapsto \left((t_1,\dots,t_n),(u:v),(\alpha:\beta)\right),$$
and is thus an automorphism (the same calculation holds on the other charts).

Since $E$ is exchanged with the strict transform of the hypersurface of equation $P_{2}+2yP_1$, which has degree $2$, passes through $p$ with multiplicity $1$, and also through a general point of $\Gamma_p$, $E$ is sent onto $2H-E-F$. Moreover, $H$ is exchanged with hyperplanes of degree $3$ having multiplicity $2$ at $p$ and multiplicity one at a general point of $\Gamma_p$. This shows that $H$ is exchanged with $3H-2E-F$. The fact that $\hat{\sigma}_p$ is an involution gives the last column of the matrix, i.e. that $F$ is exchanged with $4H-4E-F$, which corresponds to the cone $(P_{2})^2+4P_1P_3=0$, which has degree~$4$, passes through $p$ with multiplicity $4$, and through $\Gamma_p$ with multiplicity $1$.
\end{proof}

We now generalise the construction by taking many points on the same cubic hypersurface.

\begin{prop}\label{Prop:Lift}
Let $Q\subset \p^n$ be a cubic hypersurface, and let $p_1,\dots,p_k\in Q$ be distinct smooth points. For $i=1,\dots,k$ we write $\sigma_i$ the map $\sigma_{p_i,Q}$ given in Definition~$\ref{Def:InvPn}$, and by $\Gamma_{i}=\Gamma_{p_i}\subset \p^n$ the codimension $2$ subset associated.
We assume that for $p_i\notin \Gamma_{j}$ for any $i\not= j$.

Denote by $\pi\colon X\to \p^n$ the following birational morphism: it first blows-up  all points $p_1,\dots,p_k$, then blows-up the strict transform of $\Gamma_{1}$, then the strict transform of $\Gamma_{2}$ and so on until blowing-up the strict transform of $\Gamma_{k}$. 

Then, $\sigma_1,\dots,\sigma_k$ lift to pseudo-automorphisms $\hat{\sigma}_1,\dots,\hat{\sigma}_k$ of $X$.
\end{prop}
\begin{proof}
Applying Proposition~\ref{Prop:Lifts}, $\sigma_i$ lifts to an automorphism of a variety obtained by blowing-up first $p_i$ and then the strict transform of $\Gamma_i$. Since $p_j\not\in \Gamma_i$ for $j\not=i$, all these points correspond to points of the strict transform of $Q$, and are thus fixed by the lift of $\sigma_i$. We can thus blow-up all the points $p_j$ with $j\not=i$ and $\sigma_i$ again lifts to an automorphism. The strict transforms of the lifts of $\Gamma_j$ for $j\not=i$ are again contained in the strict transform of $W$ and thus fixed pointwise by the automorphism. We blow-up all these schemes and lift $\sigma_i$ to an automorphism of a variety $X_i$ obtained.

Note that $X_1=X$, but that in general $X_i$ is not the same as $X_j$ for $i\not=j$, since the varieties $\Gamma_i$ and $\Gamma_j$ intersect: the choice in the order of the sets blown-up is important. The maps $X_j\dasharrow X_i$ induced by the blow-ups are pseudo-isomorphisms, they are isomorphisms outside the pull-back of the sets where $\Gamma_i$ and $\Gamma_j$ intersect, which have codimension $\ge 2$. This implies that $\hat{\sigma_1}\in \Aut(X_1)$ and that all others $\hat{\sigma_i}$ are pseudo-automorphisms of $X$.
\end{proof}

The following proposition describes the group generated by these pseudo-auto\-morphisms, and the dynamical properties of its elements. It yields -- with Proposition~\ref{Prop:Lift} -- the proof of Theorem~\ref{Thm:CubicsPseudo}.
\begin{prop}
Let $\hat{\sigma}_1,\dots,\hat{\sigma}_k\in \Bir(X)$ be pseudo-automorphisms as in Proposition~$\ref{Prop:Lift}$. These element generate a free product  $$G=\star_{i=1}^k <\hat{\sigma}_{i}>,$$
and we have the following description of elements of $G$:

\begin{enumerate}
\item
Any element of finite order is conjugate to a $\hat{\sigma}_{i}$ and has dynamical degree~$1$;
\item
Any element conjugate to $(\hat{\sigma}_{i}\hat{\sigma}_{j})^m$ for $i\not=j$ and $m\ge 1$ is of infinite order, and its dynamical degree is equal to $1$;
\item
Any other element has dynamical degree $>1$.
\end{enumerate}
\end{prop}\begin{proof}
Let $H\in \Pic{X}$ be the pull-back of an hyperplane of $\p^n$, denote by $E_1,\dots,E_k\in \Pic{X}$ the total pull-back of the divisors obtained by blowing-up the $p_i$, and by $F_1,\dots,F_k\in \Pic{X}$ the exceptional divisors associated to $\Gamma_1,\dots,\Gamma_k$, using the same notation as before. The actions of $\hat{\sigma_1},\dots,\hat{\sigma_k}$ on $\Pic{X}$ are given by Proposition~\ref{Prop:Lifts}:
$$\begin{array}{lcrrr}
\hat{\sigma_i}(H)&=&3H&-2E_i&-F_i,\\ 
\hat{\sigma_i}(E_i)&=&2H&-E_i&-F_i,\\
 \hat{\sigma_i}(F_i)&=&4H&-4E_i&-F_i,\\
 \hat{\sigma_i}(E_j)&=&E_j&\mbox{ for }i\not=j,\\
  \hat{\sigma_i}(F_j)&=&F_j&\mbox{ for }i\not=j.\end{array}$$

Writing $\nu_i=\hat{\sigma_i}(H)-H=2H-2E_i-F_i$ we get
\begin{equation}\label{EqAppSig}\begin{array}{rcl}
\hat{\sigma_i}(H)&=&H+\nu_i,\\
\hat{\sigma_i}(\nu_i)&=&-\nu_i,\\
\hat{\sigma_i}(\nu_j)&=&\nu_j+2\nu_i\mbox{ for }i\not=j.\end{array}\end{equation}

Let us choose any element $\varphi=\sigma_{a_r}\dots \sigma_{a_1}$, where $a_1,\dots,a_r\in \{1,\dots,k\}$, $a_i\not=a_{i+1}$ for $i=1,\dots,r-1$. By induction on $r$, we prove that $\varphi(H)=H+\sum\limits_{i=1}^k \alpha_i v_i$, satisfying the following properties 
\begin{enumerate}
\item[$(i)$]
$\alpha_1,\dots,\alpha_k$ are non-negative integers;
\item[$(ii)$]
$\alpha_{a_r}> \alpha_i$ for $i\not= a_r$;
\item[$(iii)$]
if $r>1$, then $\alpha_{a_{r-1}}> \alpha_i$ for $i\notin \{a_r,a_{r-1}\}$;
\item[$(iv)$]
$\sum_{i=1}^k \alpha_i\ge (\frac{5}{3})^t$, where $t=\#\{i\ |\ i\ge 3, a_i\not=a_{i-2}\}$.
\end{enumerate}
When $r=1$, the result is obvious since $\varphi(H)=H+\nu_{a_1}$. We assume the result true for $r-1$ and prove it for $r$. We have $\varphi(H)=\sigma_{a_r}(V)$, where $V=\sigma_{a_{r-1}}\circ\dots\circ \sigma_1(H)=H+\sum\limits_{i=1}^k \beta_i v_i$ and all $\beta_i$ satisfy the properties above. In particular, $\sum\limits_{i\not=a_r} \beta_i\ge \beta_{a_{r-1}}> \beta_{a_r}$.
Applying $(\ref{EqAppSig})$, we get
$$\begin{array}{lcl}
\alpha_i&=&\beta_i\mbox{ for }i\not=a_r,\\
\alpha_{a_r}&=&1-\beta_{a_r}+2\sum\limits_{i\not=a_r} \beta_i> \sum\limits_{i\not=a_r}\beta_i\ge\beta_{a_{r-1}}>\beta_{a_r},\end{array}$$
which proves the first three assertions. To prove $(iv)$, we compute
$$\sum_{i=1}^k \alpha_i=1-4\beta_{a_r}+3 \sum_{i=1}^k \beta_i.$$

We always have $\sum\limits_{i=1}^k \alpha_i>\sum\limits_{i=1}^k \beta_i$. It suffices thus to prove that if $r\ge 3$ and $a_r\not=a_{r-2}$, then $\sum\limits_{i=1}^k \alpha_i\ge \frac{5}{3}\sum\limits_{i=1}^k \beta_i$.

 The fact that $r\ge 3$ and $a_r\not=a_{r_2}$  gives $\beta_{a_r}<\beta_{a_{r-1}}$ and $\beta_{a_r}<\beta_{a_{r-2}}$, and thus implies
$$\begin{array}{rcl}
\sum\limits_{i=1}^k \alpha_i-\frac{5}{3}\sum\limits_{i=1}^k \beta_i&=&1-4\beta_{a_r}+(3-\frac{5}{3})\sum\limits_{i=1}^k \beta_i,\\
&=&1-\frac{8}{3}\beta_{a_r}+\frac{4}{3}\sum\limits_{i\not=a_r} \beta_i,\end{array}$$
which is positive since $2\beta_{a_r}<\beta_{a_{r-1}}+\beta_{a_{r-2}}\le \sum\limits_{i\not=a_r} \beta_i$.\\

Now that $(i)-(iv)$ have been proved, we show how they imply the result. First, assertions $(i)$ and $(ii)$  show that $G$ is the free product of the groups $<\sigma_i>\cong \mathbb{Z}/2\mathbb{Z}$. Second, any non-trival element of the group is conjugated to $\varphi=\sigma_{a_r}\dots \sigma_{a_1}$, where $a_1,\dots,a_r\in \{1,\dots,k\}$, $a_i\not=a_{i+1}$ for $i=1,\dots,r-1$ and $a_r\not=a_1$. 

The element $\varphi$ has finite order if and only if $r=1$. If $r>1$, we compute its dynamical degree by computing $\deg (\varphi^n)$ for $n\in \mathbb{N}$. The degree here is the degree as a birational map of $\p^3$, which is the degree of the system $\pi(\varphi^{-n}(H))$. Since each $\nu_i$ corresponds to a divisor of degree $2$, we get
$$\deg(\varphi^n)=1+2\sum_{i=1}^k \alpha_i \mbox{ ifÊ }\varphi^{-n}(H)=H+\sum_{i=1}^k \alpha_k\nu_k.$$
The assertions above imply that if the set $\{a_1,\dots,a_r\}$ has at least three elements, $\deg(\varphi^n)\ge (\frac{5}{3})^n$, so the dynamical degree of $\varphi$ is strictly bigger than $1$. The only case where the dynamical degree $1$ could be one is when  $\varphi=(\hat{\sigma}_{i}\hat{\sigma}_{j})^m$ for $i\not=j$ and $m\ge 1$. It remains to prove that in this case, the dynamical degree is~$1$; and we only have to consider the case $m=1$. 
The submodule of $\Pic{X}$ generated by $H,\nu_i,\nu_j$ is invariant by $\varphi$, and the action relative to this basis is 

$$\left(\begin{array}{rrr}
1& 0 & 0\\
1& -1 & 2\\
0 & 0 & 1
\end{array}\right)\cdot 
\left(\begin{array}{rrr}
1& 0 & 0\\
0& 1 & 0\\
1 & 2 & -1
\end{array}\right)=
\left(\begin{array}{rrr}
1& 0 & 0\\
3& 3& -2\\
1 & 2 & -1
\end{array}\right),$$
which has only one eigenvalue, equal to $1$. This achieves the proof.\end{proof}

\begin{rema}
Note that the dynamical degree of any element of the free group $G$ generated above is easy to compute. 

$(i)$ As we observed in the above proof, the dynamical degree of $\sigma_{i}\cdot \sigma_j$, for $i\not=j$, is the biggest eigenvalue of 

$$\left(\begin{array}{rrr}
1& 0 & 0\\
1& -1 & 2\\
0 & 0 & 1
\end{array}\right)\cdot 
\left(\begin{array}{rrr}
1& 0 & 0\\
0& 1 & 0\\
1 & 2 & -1
\end{array}\right)=
\left(\begin{array}{rrr}
1& 0 & 0\\
3& 3& -2\\
1 & 2 & -1
\end{array}\right),$$
whose characteristic polynomial is $(x-1)^3$. This dynamical degree is thus $1$.

$(ii)$ We can do a similar calculation with $\sigma_{i}\cdot \sigma_j\cdot \sigma_k$ where $i,j,k$ are pairwise distinct. The dynamical degree is the the highest real eigenvalue of 

$$\left(\begin{array}{rrrr}
1& 0 & 0& 0\\
1& -1 & 2& 2\\
0 & 0 & 1& 0\\
0 & 0 & 0& 1
\end{array}\right)\cdot 
\left(\begin{array}{rrrr}
1& 0 & 0& 0\\
0 & 1& 0& 0\\
1& 2 & -1& 2\\
0 & 0 & 0& 1
\end{array}\right)\cdot 
\left(\begin{array}{rrrr}
1& 0 & 0& 0\\
0 & 1& 0& 0\\
0&  0& 1& 0\\
1 & 2 & 2& -1
\end{array}\right)=
\left(\begin{array}{rrrr}
1& 0 & 0& 0\\
9 & 15& 10& -6\\
3&  6& 3& -2\\
1 & 2 & 2& -1
\end{array}\right),$$
which is $9+4\sqrt{5}\sim 17.944272.$

$(iii)$ All other dynamical degrees can be computed in the same way.
\end{rema}
\begin{rema}
With the descriptions above, it is easy to take explicit cubic hypersurfaces, for example smooth ones, and to compute explicitly the locus to blow-up and the involutions.
\end{rema}
Now that Theorem~\ref{Thm:CubicsPseudo} is proved, we finish the section with the proof of its corollary.
\begin{proof}[proof of Corollary~$\ref{Coro1}$]
In any dimension $n\ge 3$, we take a smooth cubic hypersurface $Q\subset \p^n$, and choose three distinct general points $p_1,p_2,p_3$ such that the line through two of them intersects the cubic into another point. These points in $Q$ satisfy then the conditions of Theorem~\ref{Thm:CubicsPseudo}, and yields pseudo-automorphisms $\hat{\sigma}_1,\hat\sigma_2,\hat\sigma_3$ of the variety $X$ obtained by blowing-up $p_1,p_2,p_3$ and the varieties $\Gamma_{1},\Gamma_2,\Gamma_3$ associated. These latter being irreducible, the rank of $\Pic{X}$ is exactly $7$ (a fact which is false in dimension $2$).

Because $<\hat{\sigma}_1,\hat\sigma_2,\hat\sigma_3>$ is the free product $\star_{i=3}^k <\hat{\sigma}_{i}>$, the group generated by 

$\alpha=\hat{\sigma}_1\hat\sigma_2\hat\sigma_3$ and
$\beta=\hat\sigma_2\hat{\sigma}_1\hat\sigma_2\hat\sigma_3\hat\sigma_2$ is the free group over two generators. Moreover, none of the non-trivial elements of the group is conjugate to an element of length $<3$, so each element has dynamical degree $>1$.
\end{proof}

\section{The involutions on $\p^2$ and the blow-up}\label{Sec:Dim2}
In this section, we deal with dimension $2$.

\subsection{Degree $2$}
We will say that two birational maps $\varphi$, $\varphi'$ are \emph{projectively equivalent} if $\varphi=\alpha\varphi'\beta$ for some $\alpha,\beta\in \Aut(\p^2)$. In Question~\ref{QuesBed}, we can only study equivalence classes, since $\alpha\varphi\beta$ is conjugate to $\varphi(\beta\alpha^{-1})$. We can also replace $\varphi$ with $\varphi^{-1}$.

There are three equivalence classes of birational maps of $\p^2$ of degree $2$. Each such map has three base-points of multiplicity $1$, which are not collinear, and the classes correspond to

\begin{enumerate}
\item[(i)]
three points $p_1,p_2,p_3$ that belong to $\p^2$ as proper points;
\item[(ii)]
two points $p_1,p_2$ that belong to $\p^2$ as proper points, the point $p_3$ is infinitely near to $p_1$;
\item[(iii)]
one point $p_1$ is a proper point of $\p^2$, $p_2$ is infinitely near to $p_1$ and $p_3$ is infinitely near to $p_2$
\end{enumerate}

There are many known examples of quadratic maps of type $(i)$ that are conjugate to automorphisms of projective surfaces and have dynamical degree $>1$. See for example \cite{bib:BedKim06} or \cite{bib:BedKim09}. For examples of type $(iii)$, see \cite{bib:BedKim10}. In fact, in \cite{bib:Diller}, all possible types of quadratic maps preserving cubics and being conjugate to to automorphisms of projective surfaces and have dynamical degree $>1$ are constructed. They depend on an orbit data $[n_1,n_2,n_3]$, which provides all types $(i)$, $(ii)$, $(iii)$ depending on the number of $n_i$ equal.

All these examples yield the following:
\begin{lemm}
If $\varphi$ is a birational map of $\p^2$ of degree $2$, there exists an automorphism $\tau\in \Aut(\p^2)$  such that $\tau\varphi$ is conjugate to an automorphism of a smooth projective rational surface with dynamical degree $>1$.
\end{lemm}
\subsection{Degree $3$}

It is also possible to describe equivalence classes of elements of $\Bir(\p^2)$ of degree $3$. There are in fact $32$ algebraic families, corresponding to the type of base-points (if some are collinear, or if some infinitely near,...), or equivalently to the curves contracted (see \cite[Table of page 176]{bib:CerDesCub}). The family of biggest dimension (dimension $2$) consists of cubic maps $\varphi$ having five proper base-points, no $3$ being collinear. All others have dimension $<1$. 

We will prove that for a general cubic map $\varphi\in \Bir(\p^2)$, there exists $\tau\in \Aut(\p^2)$ such that $\tau\varphi$ is conjugate to an automorphism of positive entropy. To do this, we will use involutions $\sigma_{p,Q}$ associated to a point $p$ of a smooth cubic $\Gamma$ (see Definition~\ref{Def:InvPn}). Recall that $\sigma_{p,Q}$ preserves a general line $L$ passing through $p$ and restricts on it to the unique involution that fixes the two points $(L\cap \Gamma)\backslash\{p\}$.

The base-points of $\sigma_{p,Q}$ are described by the following lemma:
\begin{lemm} \cite[Proposition 12]{bib:CubicInertiaJ}\label{Lem:CubicInvProp}
Let $C\subset \Pn$ be a smooth cubic curve, let $p\in C$ and let $\sigma_{p,Q}$ be the element defined in Definition $\ref{Def:InvPn}$. The following occur:

$1.$ The degree of $\sigma_{p,Q}$ is $3$, and ${\sigma_{p,Q}}^2=1$, i.e. $\sigma_{p,Q}$ is a cubic involution.

$2.$ The base-points of $\sigma_{p,Q}$ are the point $p$ -- which has multiplicity $2$ -- and the four points $p_1, p_2, p_3, p_4$ such that the line passing through $p$ and $p_i$ is tangent at $p_i$ to $C$.

$3.$ If $p$ is not an inflexion point of $C$, all the points $p_1,...,p_4$ belong to $\Pn$. Otherwise, only three of them belong to $\Pn$, and the fourth is the point in the blow-up of $p$ that corresponds to the tangent of $C$ at $p$.\proofend
\end{lemm}

Since $\sigma_{p,Q}$ is an involution, the blow-up $\pi\colon X\to \p^2$ of its five base-points conjugates it to an automorphism of $X$. We now describe the action of this automorphism on the Picard group of $X$.
\begin{lemm}\label{Lem:ActionSigma}
Let $C\subset \Pn$ be a smooth cubic curve, let $p\in C$ and let $\sigma_{p,Q}$ be the element defined in Definition $\ref{Def:InvPn}$. Let $p_1,p_2,p_3,p_4$ be the base-points of $\sigma_{p,Q}$ of multiplicity one $($see Proposition~$\ref{Lem:CubicInvProp})$, and let $\pi\colon X\to \p^2$ be the blow-up of the five base-points. 

Denote by $L\subset \Pic{X}$ the pull-back of a general line of $\p^2$, by $E_i$ the divisor corresponding to the point $p_i$, and by $E$ the divisor corresponding to $p$, which is the total pull-back on $X$ of the curve contracted on $p$ $($if $p$ is an inflexion point, $E$ corresponds to a reducible curve$)$. The set $(L,E,E_1,\dots,E_4)$ is an orthogonal basis of $\Pic{X}$; the element have self-intersection $(1,-1,-1,-1,-1)$ and the action of $\hat{\sigma}=\pi^{-1}\sigma_{p,Q}\pi\in \Aut(X)$ on the Picard group is 
\[\left[\begin{array}{rrrrrr}
3 & 2 & 1& 1& 1& 1\\
-2& -1& -1 &-1 &-1 &-1\\
-1 & -1 & -1 & 0 & 0 & 0\\
-1 &  -1 & 0 & -1 & 0 & 0\\
-1 &  -1 & 0 &0 & -1 & 0\\
-1 &  -1 & 0 &0 & 0& -1 \end{array}\right]\]

\end{lemm}
\begin{proof}
Only the action of $\hat{\sigma}$ is not clear. By Lemma~\ref{Lem:CubicInvProp}, the map $\sigma_{p,Q}$ is a cubic involution, and its base-points are $p$ with multiplicity $2$, and $p_1,\dots,p_4$ with multiplicity $1$. This implies that $\hat{\sigma}(L)=3L-2E-E_1-E_2-E_3-E_4$. Because $\sigma_{p,Q}$ preserve the pencil of lines passing through $p$, we have $\hat{\sigma}(L-E)=L-E$. The lift of this pencil on $X$ gives a conic bundle $X\to \p^1$, with four singular fibres, each one being the union of $E_i$ and $L-E-E_i$ for $i=1,\dots,4$. This implies that the set $\{E_i,L-E-E_i\}$ is invariant for $i=1,\dots,4$. Computing the intersection with $L$ and $\hat{\sigma}(L)$ shows that  $\hat{\sigma}(E_i)=L-E-E_i$, for $i=1,\dots,4$. This achieves the proof.
\end{proof}

\begin{prop}\label{Prop:SigmaTau}
Let $C\subset \Pn$ be a smooth cubic curve, let $p\in C$ and let $\sigma_{p,Q}$ be the element defined in Definition $\ref{Def:InvPn}$. There exists an automorphism $\tau$ of $\p^2$,  acting via a translation of order $3$ on $C$, such that $\tau\sigma_{p,Q}$ is conjugate to an automorphism of a smooth projective rational surface $Y$, with dynamical degree $>1$.
\end{prop}
\begin{proof}
Denote as above by $p_1,\dots,p_4$ the base-points of $\sigma_{p,Q}$ of multiplicity $1$. Recall (Lemma~\ref{Lem:CubicInvProp}) that $p_1,p_2,p_3,p_4$ are the points of $C$ such that the tangent of $C$ at $p_i$ passes through $p$; if $p$ is an inflexion point, one of the points is the  point infinitely near to $p_1$ corresponding to the tangent direction of $C$.

%We prove now the existence of an automorphism $\tau$ of $\p^2$, having order $3$, preserving $C$ and acting on $C$ with no fixed point, such that $\tau(p)$ and $\tau^2(p)$ are not equal to any of the $p_i$.

We change coordinates on $\p^2$ and put the curve $C$ into its Hessian form, which is the equation 
$$x^3+y^3+z^3+\lambda xyz=0$$
for some $\lambda\in \C$ satisfying $\lambda^3\not=-27$. Let $H\subset \p^2$ be the group generated by
$$(x:y:z)\mapsto (y:z:x),$$
$$(x:y:z)\mapsto (x:\omega y:\omega^2z),$$
where $\omega$ is a $3$-rd rood of unity. One directly sees that $H$ is isomorphic to $(\mathbb{Z}/3\mathbb{Z})^2$ and preserves $C$. Moreover, the action of any of the non-trivial elements of $H$ on $C$ is fixed-point-free. We obtain thus an isomorphism of $H$ with the $3$-torsion of the group of translations $C\subset \Aut(C)$.

Let us denote by $\pi\colon Y\to \p^2$ the blow-up of the orbit of $\{p,p_1,\dots,p_4\}$ by $\tau$. If one of the $p_i$ is infinitely near to $p$, then its orbit consists of points infinitely near to the orbit of $p$. As before, we denote by $L\subset \Pic{Y}$ the pull-back of a general line of $\p^2$, by $E_i$ the divisor corresponding to the point $p_i$, and by $E$ the divisor corresponding to $p$, which is the total pull-back on $X$ of the curve contracted on $p$. 
The automorphism $\tau$ lifts to an automorphism $\hat{\tau}=\pi^{-1}\tau\pi\in \Aut(Y)$, which sends $E_i$ onto the divisor corresponding to $\tau(p_i)$, and sends $E$ onto the divisor corresponding to $\tau(p)$. 
The group $\Pic{X}$ is generated by $L$ and by $\{\hat{\tau}^i(E),\hat{\tau}^i(E_1),\dots,\hat{\tau}^i(E_4)\}_{i=0}^2$. The birational involution $\sigma_{p,Q}\in \Bir(\p^2)$ lifts to an automorphism of the surface obtained by blowing-up $p,p_1,\dots,p_4$; because this one fixes each of the other points blown-up (which belong to the curve $C$), it lifts to an automorphism $\hat{\sigma}$ of $Y$. This shows that $\tau\sigma_{p,Q}$ is conjugate by $\pi^{-1}$ to the automorphism $\hat{\tau}\hat{\sigma}$ of $Y$, and it suffices to show that its dynamical degree is $>1$ to conclude. This amounts to find a real eigenvalue of the action of $\hat{\tau}\hat{\sigma}$ on $\Pic{Y}\otimes_\mathbb{Z} \mathbb{R}$ which is bigger than $1$. The action of $\hat\sigma$ on $\Pic{Y}$ is given by Lemma~\ref{Lem:ActionSigma}, and the action of $\hat\tau$ fixes $L$ and permutes the exceptional divisors according to the action of $\tau$ on the points.

Because $\tau$ does not fix any point of $C$, the divisors $E,\hat\tau(E),\hat\tau^2(E)$ are distinct. This is also true for the divisors $E_i,\hat\tau(E_i),\hat\tau^2(E_i)$ for $i=1,\dots,4$. Note that $\tau(p_i)=p_j$ is also impossible, because it would imply that $\tau$ sends the line tangent to $C$ at $p_i$ onto the line tangent to $C$ at $p_j$, and thus $\tau$ would fix $p$. It remains to study two possible cases:

{\bf $1)$ There exists an $i\in \{1,\dots,4\}$ such that $\tau(p)=p_i$ or $\tau^2(p)=p_i$.}\\
Replacing $\tau$ by $\tau^2$ and renumbering the $p_i$ if needed, we can assume that $\tau(p)=p_1$. We see that $\tau(p_1)=\tau^2(p)$ is distinct from $\tau^i(p_j)$ for $j\ge 2$. The sub-$\mathbb{Z}$-module of $\Pic{Y}$ generated by $$L, E, \hat\tau(E)=E_1,\hat\tau(E_1),\sum_{i=2}^4E_i, \sum_{i=2}^4\hat\tau(E_i), \sum_{i=2}^4\hat\tau^2(E_i)$$
is invariant, and the action of $\hat{\tau}$ and $\hat{\sigma}$, relatively to this basis, are given by

\[\left[\begin{array}{rrrrrrr}
3 & 2 & 1& 0& 3& 0&0\\
-2& -1& -1 &0 &-3 &0 & 0\\
-1 & -1 & -1 & 0 & 0 & 0&0\\
0 & 0&0 & 1& 0 & 0&0\\
-1 &  -1 & 0 & 0 & -1 & 0&0\\
0 &  0 & 0 &0 & 0 & 1&0\\
0 &  0 & 0 &0 & 0& 0&1 \end{array}\right],
\left[\begin{array}{rrrrrrr}
1 & 0 & 0&  0& 0& 0&0\\
0&  0&  0  & 1 &0 &0 & 0\\
0 & 1 & 0  &0 & 0 & 0&0\\
0 & 0&  1 & 0& 0 & 0&0\\
0 &  0 & 0 & 0 & 0 & 0&1\\
0 &  0 & 0 &0 & 1 & 0&0\\
0 &  0 & 0 &0 & 0& 1&0 \end{array}\right].\]

The action of $\hat\tau\hat\sigma$ is thus the product of the two matrices, which is
\[\left[\begin{array}{rrrrrrr}
3&2&1&0&3&0&0\\ 0&0&0&1&0&0&0\\ -2&-1&-1&0&-3&0&0\\ -1&-1&-1&0&0&0&0\\ 0&0&0&0&0&0&1\\ -1&-1&0&0&-1&0&0\\ 0&0&0&0&0&1&0 \end{array}\right].\]
The characteristic polynomial is $x^7-2x^6+2x-1=(x-1)(x^6-x^5-x^4-x^3-x^2+1)$, whose real eigenvalues are $\lambda, 1,\lambda^{-1}$, where $\lambda\sim 1.946856$. This number is the dynamical degree of $\hat\tau\hat\sigma$ (and also of $\tau\sigma_{p,Q}$).

{\bf $2)$ For $i=1,\dots,4$,  $p_i\not\in\{\tau(p),\tau^2(p)\}$.}\\
In this case, the sub-$\mathbb{Z}$-module of $\Pic{Y}$ generated by 
$$L,2E+\sum_{i=1}^4 E_i, 2\hat\tau(E)+\sum_{i=1}^4\hat\tau(E_i), 2\hat\tau^2(E)+\sum_{i=1}^4\hat\tau^2(E_i)$$
is invariant, and the actions of $\hat{\sigma}$ and $\hat{\tau}$ are given by

\[\left[\begin{array}{rrrr}
3 & 8 & 0& 0\\
-1& -3& 0 &0 \\
0 & 0 & 1 & 0 \\
0 & 0&0 & 1  \end{array}\right],
\left[\begin{array}{rrrr}
1 & 0 & 0&  0\\
0&  0&  0  & 1 \\
0 & 1 & 0  &0 \\
0 & 0&  1 & 0 \end{array}\right].\]

The action of $\hat\tau\hat\sigma$ is thus the product of the two matrices, which is

\[\left[\begin{array}{rrrr}
3 & 8 & 0& 0\\
0& 0& 0 &1 \\
-1 & -3 & 0 & 0 \\
0 & 0&1 & 0  \end{array}\right].\]
The characteristic polynomial is $x^4-3x^3+3x-1=(x-1)(x+1)(x^2-3x+1)$, whose real eigenvalues are $-1,1,\frac{3\pm \sqrt{5}}{2}$. The dynamical degree is then $\frac{3+ \sqrt{5}}{2}\sim 2.618034$.
\end{proof}
\begin{coro}
Let $\varphi\in \Bir(\p^2)$ be a birational map of degree $3$. 

\begin{enumerate}
\item
If all base-points of $\varphi$ and $\varphi^{-1}$ are proper points of the plane, there exists an automorphism $\tau\in \Aut(\p^2)$ such that $\tau\varphi$ is conjugate to an automorphism of a smooth projective rational surface with dynamical degree $>1$.
\item
In the algebraic set of birational maps of $\p^2$ of degree $3$, the set of maps having this property is a dense subset with complement of codimension $1$.
\end{enumerate}
\end{coro}
\begin{proof}
$(1)$ We can replace $\varphi$ with $\alpha\varphi\beta$, where $\alpha,\beta\in \Aut(\p^2)$. In particular, we can assume that the base-point of multiplicity $2$ of both $\varphi$ and $\varphi^{-1}$ is $p=(1:0:0)$ and that a general line passing through this point is invariant by $\alpha$. The other base-points of $\varphi$ are $p_2,p_3,p_4,p_5$. Note that no $2$ of the $p_i$ are collinear with $p$, because otherwise the linear system of $\varphi$ (being cubics of degree $3$, with multiplicity $2$ at $p$ and multiplicity $1$ at $p_1,\dots,p_5$) would be reducible. If three of the $p_i$ are collinear, the line passing through these points would have self-intersection $-2$ on the blow-up of $p,p_1,\dots,p_5$, so the map $\varphi^{-1}$ would have a base-point being infinitely near (to~$p$). This implies that no $3$ of the points $p,p_1,\dots,p_5$ are collinear. Choosing the good automorphism $\beta$, we can then assume that 
\begin{center}$p=(1:0:0)$, $p_1=(0:1:0)$, $p_2=(0:0:1)$, $p_3=(1:1:1)$, $p_4=(a:b:c)$,\end{center}
for some $a,b,c\in \C^{*}$, no two being equal. We consider the birational cubic involution 
\begin{center}$\begin{array}{lrccccccl}
\sigma\colon (x:y:z) \dasharrow \big(\h{3}&-ayz(&\h{3}(c-b)x&\h{2}+\h{2}&(a-c)y&\h{2}+\h{2}&(b-a)z&\h{3})&\h{2}:\\ 
\h{3}& y(&\h{3}a(c-b)yz&\h{2}+\h{2}&b(a-c)xz&\h{2}+\h{2}&c(b-a)xy&\h{3})&\h{2}:\\
\h{3}& z(&\h{3}a(c-b)yz&\h{2}+\h{2}&b(a-c)xz&\h{2}+\h{2}&c(b-a)xy&\h{3})&\h{2}\big),\end{array}$\end{center}
and observe that its base-points are $p_1,\dots,p_4$, and $p$ with multiplicity $2$. It preserves a general line passing through $p=(1:0:0)$. Moreover it fixes pointwise the smooth cubic curve $C\subset \p^2$ of equation
\begin{center} $b(a-c)x^2z+c(b-a)x^2y+a(a-c)y^2z+a(b-a)yz^2+2a(c-b)xyz=0.$\end{center}
In particular, the map $\sigma$ is equal to the involution $\sigma_{p,Q}$ associated to $p\in C$, according to Definition~$\ref{Def:InvPn}$. Because $\sigma$ and $\varphi$ have the same linear system (same degree, same base-points with same multiplicities), then $\varphi$ is equal to $\sigma\gamma$ for some $\gamma\in \Aut(\p^2)$. Assertion $(1)$ follows then from Proposition~\ref{Prop:SigmaTau}. Note that the existence of $\sigma$ (which is uniquely determined by $p$, $p_1,\dots,p_4$) can also be seen more geometrically, by looking at the automorphism group of the del Pezzo surface of degree $4$ obtained by blowing-up $p,p_1,\dots,p_4$, (see \cite[Lemma~9.11]{bib:BlaLinearisation}).

It remains to prove Assertion $(2)$. Any cubic birational map $\varphi$ of $\p^2$ has one base-point of multiplicity $2$ and four base-points of multiplicity $1$. And two maps with the same base-points differ only by the post-composition with an automorphism. The set of cubic birational maps is then parametrised by one point of $\p^2$, a set of four other points, that are on $\p^2$ or infinitely near, and one automorphism of $\p^2$. The biggest dimension is when all points are on $\p^2$ and no $3$ are collinear, which is exactly the set where the map and its inverse have only proper base-points. It has the dimension of $(\p^2)^5\times \PGL(3,\C)$, which is $18$. The set of all other maps has only dimension $17$; it corresponds to the cases where $3$ points are collinear or one point is infinitely near.
\end{proof}

\begin{rema}
$i)$ By Proposition~$\ref{Prop:SigmaTau}$, the same result holds for map projectively equivalent to $\sigma_{p,Q}$ where $p$ is an inflexion point of a smooth cubic $Q$. These are the maps having $4$ proper base-points $p,p_1,p_2,p_3$ of multiplicity $2,1,1,1$ such that $p_1,p_2,p_3$ are collinear,  and a point $p_5$ infinitely near to $p$.

$ii)$ It also holds for other special cubics maps: some with two proper base-points  $($see \cite{bib:BedDil06}$)$, with one proper base-point $($see \cite{bib:BedKim10}$)$, or some maps of degree $3$ with exactly two proper base-points $($see \cite{bib:BedDil06} and \cite{bib:DeGr}$)$.

$iii)$ If a birational map $\varphi\in \Bir(\p^2)$ of degree $3$ has all its base-points which are proper but $3$ are collinear, then it is projectively equivalent to 
$$\psi_\lambda\colon(x:y:z)\dasharrow (yz(y-z+(\lambda-1)x):xy(z-\lambda y):xz(z-\lambda y))$$
for some $\lambda\in \C\setminus\{0,1\}$. In particular, $\varphi^{-1}$ has only $4$ proper base-points.
\end{rema}
\begin{ques}
Does there exists a $\lambda\in \C\setminus\{0,1\}$ such that for any $\tau\in \Aut(\p^2)$ the map $\tau\psi_\lambda$ is not conjugate to an automorphism of a projective surface $($with dynamical degree $>1)$?
\end{ques}
 \section{The example}\label{Sec:Example6}
 
\subsection{Actions on infinitely near points} Before proving Theorem~\ref{Thm:Exa6}, we need some general tools.

Let $X,Y$ be two smooth projective rational surfaces, and let $\psi\colon X\dasharrow Y$ be a birational map.
If $p$ is a point of $X$ or a point infinitely near, which is not a base-point of $\psi$, we define 
a point~$\psi^\bullet (p)$, which will also be a point of $Y$ or a point infinitely near. For this, take a minimal 
resolution 
$$
\xymatrix{& Z\ar[rd]^{\pi_2}\ar[ld]_{\pi_1}&\\
X\ar@{-->}[rr]_{\psi}&&Y,
}$$
where $\pi_1$, $\pi_2$ are sequences of blow-ups. Because $p$ is not a base-point of $\psi$, it corresponds, via $\pi_1$, to a point 
of $\mathrm{Z}$ or infinitely near. Using $\pi_2$, we view this point on $Y$, again maybe infinitely near, and call it 
$\psi^\bullet(p)$. Observe that this point is not a base-point of $\psi^{-1}$, and that $(\psi^{-1})^{\bullet}(\psi^\bullet(p))=p$.

\begin{rema}
If $p$ is not a base-point of $\phi\in \Bir(X)$ and $\phi(p)$ is not a base-point of $\psi\in \Bir(X)$, we have 
$(\psi\phi)^{\bullet}(p)=\psi^\bullet(\phi^\bullet(p))$. If $p$ is a general point of $X$, then $\phi^{\bullet}(p)=
\phi(p)\in X$.
\end{rema}

\begin{exem} If  $p=(1:0:0)\in \p^2$ and $\psi\in \Bir(\p^2)$ is the map $(x:y:z)\dasharrow (yz+x^2:xz:z^2),$
the point $\psi^{\bullet}(p)$ is not equal to $p=\psi(p)$, but is infinitely near to it.
\end{exem}

The following easy lemma will be important for the proof of Theorem~\ref{Thm:Exa6}.
\begin{lemm}
Let $\varphi\in \Bir(\p^2)$ be a birational map and let $p$ be a point of $\p^2$ $($or infinitely near$)$. If there exists  an integer $N\ge 0$ such that $p$ is a base-point of $\varphi^{-k}$ for any $k\geq N$ but is not a base-point of $\varphi^{k}$ for any $k\ge N$, then $\varphi$ is not conjugate to an automorphism of a smooth projective surface.\label{Lem:PtsBasesNotCOnj}
\end{lemm}
\begin{proof}
We prove first that $(\varphi^{k})^{\bullet}(p)$ and $(\varphi^{l})^{\bullet}(p)$ are two distinct points of $\p^2$ (or infinitely near), for any $k,l\ge N$ with $k\not=l$. Otherwise, assuming that $l> k$, the equality $(\varphi^{k})^{\bullet}(p)=(\varphi^{l})^{\bullet}(p)$  implies that $\varphi^{-l}$ is defined at $(\varphi^{k})^{\bullet}(p)$ (because it is defined at $(\varphi^{l})^{\bullet}(a)$), and that $(\varphi^{-l})^{\bullet}((\varphi^{k})^{\bullet}(p))=p$. In particular, $\varphi^{k-l}$ is defined at $p$, and $(\varphi^{k-l})^{\bullet}(p)=p$, which means that $(\varphi^{(k-l)m})^{\bullet}(p)=p$ for any $m\ge 0$. This is incompatible with the fact that $p$ is a base-point of $\varphi^{-i}$  for any $i\ge N$.

The set $\{(\varphi^{k})^{\bullet}(p)\}_{k=N}^{\infty}$ is thus an infinite set of points that belong to $\p^2$, as proper or infinitely near points. Suppose now that there exists a birational map $\alpha \colon \p^2\dasharrow S$, where $S$ is a smooth projective surface, that conjugates $\varphi$ to an automorphism $\hat{\varphi}=\alpha\varphi\alpha^{-1}$ of $S$. The map $\alpha$ having only a finite number of base-points, there exists $M\ge N$ such that no one of the points $\{(\varphi^{k})^{\bullet}(p)\}_{k=M}^{\infty}$ is a base-point of $\alpha$. Writing $p_k=\alpha^\bullet((\varphi^{k})^{\bullet}(p))$ for any $k\ge M$, we obtain a family of distinct points $\{p_k\}_{k=M}^\infty$ such that $\hat{\varphi}(p_k)=p_{k+1}$ for each $k\ge M$. Writing $p_{k-m}=\hat{\varphi}^{-m}(p_k)$ for any $m\ge 0$, we obtain an orbit $\{p_k\}_{k\in \mathbb{Z}}$ of the automorphism $\hat{\varphi}$, so that $p_k\not= p_l$ for $k\not=l$. Increasing maybe $M$, we can assume that $p_k$ is not a base-point of $\alpha^{-1}$ for any $k\ge M$ and any $k\le -M$.
This implies that  $(\varphi^{M})^{\bullet}(p)$ is not a base-point of the map $\varphi^{-2M}=\alpha^{-1} \hat{\varphi}^{2M} \alpha$; indeed, $\alpha^{\bullet}((\varphi^{M})^{\bullet}(p))=p_m$, and $(\hat{\varphi}^{-2m})^{\bullet}(p_m)=p_{-m}$, which is not a base-point of  $\alpha^{-1}$.  

We obtain a contradiction with the fact that $p$ is a base-point of $\varphi^M$ but not of $\varphi^{-M}$.
\end{proof}

 The section is devoted to the proof of Theorem~\ref{Thm:Exa6}. We will always denote by $\chi\in \Bir(\p^2)$ the birational map
$$\chi\colon (x:y:z)\dasharrow (xz^5+(yz^2+x^3)^2:yz^5+x^3z^3:z^6)$$
which restricts on the affine plane where $z=1$ to the automorphism
$$(x,y)\mapsto (x+(y+x^3)^2,y+x^3),$$ being the composition of 
 $(x,y)\mapsto (x+y^2,y)$ with $(x,y)\mapsto (x,y+x^3)$.
 
 \subsection{Basic description of the map $\chi$}
 The proof of Theorem~\ref{Thm:Exa6} will rely on the study of the base-points of $\chi$, and of its inverse, which is
 $$\chi^{-1}\colon (x:y:z)\mapsto (xz^5-y^2z^4:yz^5-(xz-y^2)^3:z^6).$$
 It will use two main properties: both $\chi$ and $\chi^{-1}$ have only one proper base-point, but the geometry of the base-points of the two maps are different (see below for more details). Note that many other examples can be constructed in the same way, the map $\chi$ is only the simplest one having the above properties.
 
 \subsection{Base-points of $\chi$}
 As all birational maps of $\p^2$ which contract only one curve, $\chi$ has only one proper base-point, namely $p_1=(0:1:0)$, and all its base-points are "in tower" (see \cite[Lemma 2.2]{bib:BlaCrelle}). This means that the $8$ base-points of $\chi$, that we denote by $p_1,\dots,p_8$, are such that $p_i$ is infinitely near to $p_{i-1}$ for $i=2,\dots,8$. We denote by $\pi\colon X\to \p^2$ the blow-up of the $8$ base-points, and write $\mathcal{C}\subset X$ the strict transform of the line $C\subset \p^2$ of equation $z=0$, which is the only curve of $\p^2$ contracted by $\chi$. We denote by $\mathcal{E}_i\subset X$ the strict transform of the curve obtained by blowing-up $p_i$. The intersection form on $X$ corresponds to the dual diagram of Figure~\ref{ConfigX} (this can be checked directly in local charts or by the decomposition of $\chi$ into two simple maps as above).
 
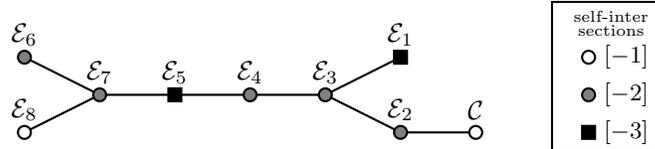
\begin{figure}[ht]\begin{pspicture}(-1,-0.3)(8,0.8)%taille
\rput(-1,0.8){{$\mathcal{E}_6$}}
\rput(-1,-0.2){{$\mathcal{E}_8$}}
\rput(0,0.3){{$\mathcal{E}_7$}}
\rput(1,0.3){{$\mathcal{E}_5$}}
\rput(2,0.3){{$\mathcal{E}_4$}}
\rput(3,0.3){{$\mathcal{E}_3$}}
\rput(4,0.8){{$\mathcal{E}_1$}}
\rput(4,-0.2){{$\mathcal{E}_2$}}
\rput(5,-0.2){{$\mathcal{C}$}}
\psline(0,0)(3,0)
\psline(-1,0.5)(0,0)
\psline(-1,-0.5)(0,0)
\psline(3,0)(4,0.5)
\psline(3,0)(4,-0.5)
\psline(4,-0.5)(5,-0.5)
\pscircle[fillstyle=solid,fillcolor=gray](-1,0.5){0.1}
\pscircle[fillstyle=solid,fillcolor=white](-1,-0.5){0.1}
\pscircle[fillstyle=solid,fillcolor=gray](0,0){0.1}
\psframe[fillstyle=solid,fillcolor=black](0.9,-0.1)(1.1,0.1)
\pscircle[fillstyle=solid,fillcolor=gray](2,0){0.1}
\pscircle[fillstyle=solid,fillcolor=gray](3,0){0.1}
\pscircle[fillstyle=solid,fillcolor=gray](4,-0.5){0.1}
\pscircle[fillstyle=solid,fillcolor=white](5,-0.5){0.1}
\psframe[fillstyle=solid,fillcolor=black](3.9,0.4)(4.1,0.6)
\psframe(6,-0.75)(7.5,1.25)
\psframe[fillstyle=solid,fillcolor=black](6.4,-0.6)(6.6,-0.4)
\pscircle[fillstyle=solid,fillcolor=white](6.5,0.5){0.1}
\pscircle[fillstyle=solid,fillcolor=gray](6.5,0){0.1}
\rput(6.8,1.05){{\tiny self-inter}}
\rput(6.8,0.85){{\tiny sections}}
\rput(7,0.5){\small {$[-1]$}}
\rput(7,0){\small{$[-2]$}}
\rput(7,-0.5){\small{$[-3]$}}
\end{pspicture}\\
\caption{\small The configuration of the curves $\mathcal{E}_1,\dots,\mathcal{E}_8,\mathcal{C}$ on the surface~$X$\label{ConfigX}. Two curves are connected by an edge if their intersection is positive (and here equal to $1$). The self intersections correspond to the shape of the vertices.}\end{figure}

\bigskip

Let us write $\varphi=\tau\chi$, where $\tau$ is an automorphism of $\p^2$. 
Because $\pi$ is the blow-up of the base-points of $\chi$, which are also the base-points of $\varphi$, the map $\eta=\varphi\pi$ is a birational morphism $X\to \p^2$, which is the blow-up of the base-points of $\varphi^{-1}$. In fact, Figure~\ref{Resolutionvarphi} is the minimal resolution of $\varphi$.

\begin{figure}[ht]\centerline{
\xymatrix{&  X\ar[ld]_{\pi}\ar[rd]^{\eta}\\
\p^2\ar@{-->}[rr]_{\varphi}&&\p^2}
}
\caption{\small \label{Resolutionvarphi}The minimal resolution of indeterminacies of $\varphi$.}\end{figure}
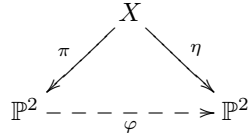

\bigskip

The line $C\subset\p^2$ being the only  curve of $\p^2$ contracted by $\varphi$, the morphism $\eta$ contracts $\mathcal{C}$, and the union of $7$ other irreducible curves, which are among the curves $\mathcal{E}_1,\dots,\mathcal{E}_8$. The configuration of  Figure~\ref{ConfigX} shows that $\eta$ contracts the curves 
$$\mathcal{C}, \mathcal{E}_2, \mathcal{E}_3, \mathcal{E}_4, \mathcal{E}_1, \mathcal{E}_5, \mathcal{E}_7,\mathcal{E}_6$$
following this order. 

We study now the blow-up $\eta\colon X\to \p^2$ in the same way as we did for $\pi$. We denote by $q_1,\dots,q_8$ the base-points of $\varphi^{-1}$ (or equivalently the points blown-up by~$\eta$), so that $q_1=(1:0:0)$, and $q_i$ is infinitely near to $q_{i-1}$ for $i=2,\dots,8$. We denote by $D\subset \p^2$ the line which is contracted by $\varphi^{-1}$ (which is the image by $\tau$ of the line of equation $z=0$), and write $\mathcal{D}\subset X$ the strict transform by $\eta^{-1}$ of the curve $D$. We then denote by $\mathcal{F}_i$ the strict transform of the curve obtained by blowing-up $q_i$. Because of the order of the curves contracted by $\eta$, we get equalities between $\mathcal{C},\mathcal{E}_1,\dots,\mathcal{E}_8$ and $\mathcal{D},\mathcal{F}_1,\dots,\mathcal{F}_8$, according to Figure~\ref{ConfigX2}.

\begin{figure}[ht]\begin{pspicture}(-1,-0.6)(7,0.9)%taille
\rput(-1,0.8){\tiny {$\mathcal{E}_6\!=\!\mathcal{F}_1$}}
\rput(-1,-0.2){\tiny{$\mathcal{E}_8\!=\!\mathcal{D}$}}
\rput(0.2,-0.2){\tiny{$\mathcal{E}_7\!=\!\mathcal{F}_2$}}
\rput(1,0.3){\tiny{$\mathcal{E}_5\!=\!\mathcal{F}_3$}}
\rput(2,-0.2){\tiny{$\mathcal{E}_4\!=\!\mathcal{F}_5$}}
\rput(3,0.3){\tiny{$\mathcal{E}_3\!=\!\mathcal{F}_6$}}
\rput(4,0.8){\tiny{$\mathcal{E}_1\!=\!\mathcal{F}_4$}}
\rput(4,-0.7){\tiny{$\mathcal{E}_2\!=\!\mathcal{F}_7$}}
\rput(5,-0.2){\tiny{$\mathcal{C}\!=\!\mathcal{F}_8$}}
\psline(0,0)(3,0)
\psline(-1,0.5)(0,0)
\psline(-1,-0.5)(0,0)
\psline(3,0)(4,0.5)
\psline(3,0)(4,-0.5)
\psline(4,-0.5)(5,-0.5)
\pscircle[fillstyle=solid,fillcolor=gray](-1,0.5){0.1}
\pscircle[fillstyle=solid,fillcolor=white](-1,-0.5){0.1}
\pscircle[fillstyle=solid,fillcolor=gray](0,0){0.1}
\psframe[fillstyle=solid,fillcolor=black](0.9,-0.1)(1.1,0.1)
\pscircle[fillstyle=solid,fillcolor=gray](2,0){0.1}
\pscircle[fillstyle=solid,fillcolor=gray](3,0){0.1}
\pscircle[fillstyle=solid,fillcolor=gray](4,-0.5){0.1}
\pscircle[fillstyle=solid,fillcolor=white](5,-0.5){0.1}
\psframe[fillstyle=solid,fillcolor=black](3.9,0.4)(4.1,0.6)
\psframe(6,-0.75)(7.5,1.25)
\psframe[fillstyle=solid,fillcolor=black](6.4,-0.6)(6.6,-0.4)
\pscircle[fillstyle=solid,fillcolor=white](6.5,0.5){0.1}
\pscircle[fillstyle=solid,fillcolor=gray](6.5,0){0.1}
\rput(6.8,1.05){{\tiny self-inter}}
\rput(6.8,0.85){{\tiny sections}}
\rput(7,0.5){\small {$[-1]$}}
\rput(7,0){\small{$[-2]$}}
\rput(7,-0.5){\small{$[-3]$}}
\end{pspicture}\\
\caption{\small The configuration of the curves $\mathcal{E}_i,\mathcal{F}_i,\mathcal{C},\mathcal{D}$ on the surface~$X$\label{ConfigX2}.}\end{figure}
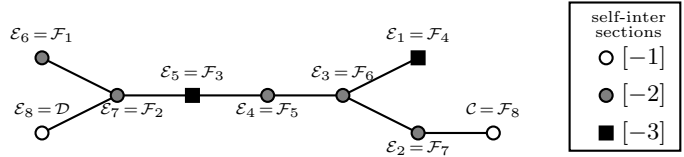

In particular, we see that the configuration of the points $p_1,\dots,p_8$ is not the same as the one of the points $q_1,\dots,q_8$. Saying that a point  $a$ is \emph{proximate to} a point $b$ if $a$ is infinitely near to $b$ and belongs to the strict transform of the curve obtained by blowing-up $b$, the configuration of the points $p_i$ and of the points $q_i$ are given in Figure~\ref{Configpts}.
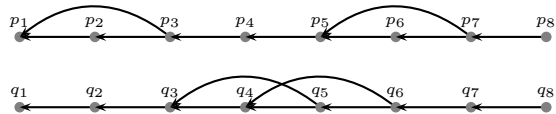
\begin{figure}[ht]\begin{pspicture}(0,-0.2)(9,0.7)%taille
\rput(1,0.2){\tiny {$p_1$}}
\rput(2,0.2){\tiny {$p_2$}}
\rput(3,0.2){\tiny {$p_3$}}
\rput(4,0.2){\tiny {$p_4$}}
\rput(5,0.2){\tiny {$p_5$}}
\rput(6,0.2){\tiny {$p_6$}}
\rput(7,0.2){\tiny {$p_7$}}
\rput(8,0.2){\tiny {$p_8$}}
\pscircle[fillstyle=solid,fillcolor=gray,linecolor=gray](1,0){0.07}
\psline{->}(2,0)(1,0)
\pscircle[fillstyle=solid,fillcolor=gray,linecolor=gray](2,0){0.07}
\psline{->}(3,0)(2,0)
\parabola{<-}(1,0)(2,0.4)
\pscircle[fillstyle=solid,fillcolor=gray,linecolor=gray](3,0){0.07}
\psline{->}(4,0)(3,0)
\pscircle[fillstyle=solid,fillcolor=gray,linecolor=gray](4,0){0.07}
\psline{->}(5,0)(4,0)
\pscircle[fillstyle=solid,fillcolor=gray,linecolor=gray](5,0){0.07}
\psline{->}(6,0)(5,0)
\pscircle[fillstyle=solid,fillcolor=gray,linecolor=gray](6,0){0.07}
\psline{->}(7,0)(6,0)
\parabola{<-}(5,0)(6,0.4)
\pscircle[fillstyle=solid,fillcolor=gray,linecolor=gray](7,0){0.07}
\psline{->}(8,0)(7,0)
\pscircle[fillstyle=solid,fillcolor=gray,linecolor=gray](8,0){0.07}
\end{pspicture}\\
\begin{pspicture}(0,-0.2)(9,0.7)%taille
\rput(1,0.2){\tiny {$q_1$}}
\rput(2,0.2){\tiny {$q_2$}}
\rput(3,0.2){\tiny {$q_3$}}
\rput(4,0.2){\tiny {$q_4$}}
\rput(5,0.2){\tiny {$q_5$}}
\rput(6,0.2){\tiny {$q_6$}}
\rput(7,0.2){\tiny {$q_7$}}
\rput(8,0.2){\tiny {$q_8$}}
\pscircle[fillstyle=solid,fillcolor=gray,linecolor=gray](1,0){0.07}
\psline{->}(2,0)(1,0)
\pscircle[fillstyle=solid,fillcolor=gray,linecolor=gray](2,0){0.07}
\psline{->}(3,0)(2,0)
\pscircle[fillstyle=solid,fillcolor=gray,linecolor=gray](3,0){0.07}
\parabola{<-}(3,0)(4,0.4)
\psline{->}(4,0)(3,0)
\pscircle[fillstyle=solid,fillcolor=gray,linecolor=gray](4,0){0.07}
\parabola{<-}(4,0)(5,0.4)
\psline{->}(5,0)(4,0)
\pscircle[fillstyle=solid,fillcolor=gray,linecolor=gray](5,0){0.07}
\psline{->}(6,0)(5,0)
\pscircle[fillstyle=solid,fillcolor=gray,linecolor=gray](6,0){0.07}
\psline{->}(7,0)(6,0)
\pscircle[fillstyle=solid,fillcolor=gray,linecolor=gray](7,0){0.07}
\psline{->}(8,0)(7,0)
\pscircle[fillstyle=solid,fillcolor=gray,linecolor=gray](8,0){0.07}
\end{pspicture}
\caption{\small The configuration of the points $p_1,\dots,p_8$ and of the points $q_1,\dots,q_8$. An arrow corresponds to the relation "is proximate to".\label{Configpts}}\end{figure}

\subsection{The proof of the theorem}
\begin{proof}[proof of Theorem~$\ref{Thm:Exa6}$]
We write as above $\varphi=\tau\chi$, where $\tau$ is an automorphism of $\p^2$, and recall that $p_1,\dots,p_8$ are the base-points of $\varphi$, and $q_1,\dots,q_8$ are the base-points of $\varphi^{-1}$. Our aim is to show that $p_3$ is a base-point of $\varphi^i$ and not of $\varphi^{-i}$, for any $i>0$. This will imply that $\varphi$ is not conjugate to an automorphism of a smooth projective rational surface, by Lemma~\ref{Lem:PtsBasesNotCOnj}.  

Denote by $k$ the lowest positive integer such that $p_1$ is a base-points of $\varphi^{-k}$; if no such integer exists, we write $k=\infty$.

For any integer $i$ such that $1\le i<  k$, the point $p_1$ is not a base-point of $\varphi^{-i}$, hence the maps $\varphi$ and $\varphi^{-i}$ have no common base-point. This implies that the set of base-points of the map $\varphi^{i+1}=\varphi\circ\varphi^{i}$ is the union of the base-points of $\varphi^{i}$ and of the points $(\varphi^{-i})^\bullet(p_j)$ for $j=1,\dots,8$. The map $\varphi^{-i}$ being defined at $p_1$, the point $(\varphi^{-i})^\bullet(p_j)$  is proximate to the point $(\varphi^{-i})^\bullet(p_k)$ if and only if $p_j$ is proximate to $p_k$.

Proceeding by induction on $i$, we obtain the following assertions:

\begin{enumerate}
\item
For any integer $i$ with $1\le i \le k$, the set of base-points of $\varphi^i$  is equal to
$$ \{(\varphi^{-m})^\bullet(p_j)\ |\ j=1,\dots,8, m=0,\dots,i-1\}.$$
\item
For any integer $l$ with $0\le -l<k$, the configuration of the points $\{(\varphi^{l})^\bullet(p_j)\}_{j=1}^8$ is given by 

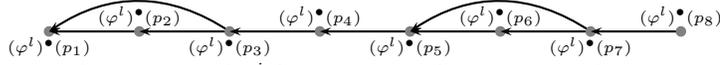
\begin{pspicture}(0,-0.2)(9,0.7)%taille
\psset{xunit=1.2}
\rput(1,-0.2){\tiny {$(\varphi^{l})^\bullet(p_1)$}}
\rput(2,0.2){\tiny {$(\varphi^{l})^\bullet(p_2)$}}
\rput(3,-0.2){\tiny {$(\varphi^{l})^\bullet(p_3)$}}
\rput(4,0.2){\tiny {$(\varphi^{l})^\bullet(p_4)$}}
\rput(5,-0.2){\tiny {$(\varphi^{l})^\bullet(p_5)$}}
\rput(6,0.2){\tiny {$(\varphi^{l})^\bullet(p_6)$}}
\rput(7,-0.2){\tiny {$(\varphi^{l})^\bullet(p_7)$}}
\rput(8,0.2){\tiny {$(\varphi^{l})^\bullet(p_8)$}}
\pscircle[fillstyle=solid,fillcolor=gray,linecolor=gray](1,0){0.07}
\psline{->}(2,0)(1,0)
\pscircle[fillstyle=solid,fillcolor=gray,linecolor=gray](2,0){0.07}
\psline{->}(3,0)(2,0)
\parabola{<-}(1,0)(2,0.4)
\pscircle[fillstyle=solid,fillcolor=gray,linecolor=gray](3,0){0.07}
\psline{->}(4,0)(3,0)
\pscircle[fillstyle=solid,fillcolor=gray,linecolor=gray](4,0){0.07}
\psline{->}(5,0)(4,0)
\pscircle[fillstyle=solid,fillcolor=gray,linecolor=gray](5,0){0.07}
\psline{->}(6,0)(5,0)
\pscircle[fillstyle=solid,fillcolor=gray,linecolor=gray](6,0){0.07}
\psline{->}(7,0)(6,0)
\parabola{<-}(5,0)(6,0.4)
\pscircle[fillstyle=solid,fillcolor=gray,linecolor=gray](7,0){0.07}
\psline{->}(8,0)(7,0)
\pscircle[fillstyle=solid,fillcolor=gray,linecolor=gray](8,0){0.07}
\end{pspicture}

\end{enumerate}
In particular, $p_3$ is a base-point of $\varphi^i$ for any $i$ satisfying $1\le i\le k$.

If $k=\infty$, this implies that $p_3$ is a base-point of $\varphi^{i}$ for any $i>0$, and by definition of $k$, the point $p_1$ is not a base-point of $\varphi^{-i}$ for any $i>0$, so neither is $p_3$. We can thus assume that $k$ is a positive number.

Observe that $q_1$ is not a base-point of $\varphi^i$ for $1\le i\le k-1$. Indeed, otherwise $q_1$ would be equal to $(\varphi^{-m})^\bullet(p_j)$ for some $m,j$ satisfying $0\le m\le k-2$ and $1\le j\le 8$. This would imply that $p_j$ is a base-point of $\varphi^{m+1}$, which is impossible. We see thus that $\varphi^{-1}$ has no common base-point with $\varphi^i$ for $1\le i\le k-1$. In particular, the set of common base-points of $\varphi^{-1}$ and $\varphi^k$ is equal to $$B=\{(\varphi^{-(k-1)})^{\bullet}(p_j)\}_{j=1}^8\cap \{q_j\}_{j=1}^8.$$
 
Because $p_1$ is a base-point of $\varphi^{-k}$ and not of $\varphi^{-(k-1)}$, the point $(\varphi^{-(k-1)})^{\bullet}(p_1)$, which is a base-point of $\varphi^k$, is also a base-point of $\varphi^{-1}$. In particular, the set $B$ is not empty. The configuration of the two sets of points $\{(\varphi^{-(k-1)})^{\bullet}(p_j)\}_{j=1}^8$ and $\{q_j\}_{j=1}^8$ implies that $q_1=(\varphi^{-(k-1)})^{\bullet}(p_1)$. Moreover, either $B=\{q_1\}$ or $B=\{q_1,q_2\}$. Indeed, the point $q_3$ is proximate to $q_1$ but not to $q_2$, whereas $(\varphi^{-(k-1)})^{\bullet}(p_3)$ is proximate to $(\varphi^{-(k-1)})^{\bullet}(p_1)$ and $(\varphi^{-(k-1)})^{\bullet}(p_2)$.

The point $(\varphi^{-(k-1)})^{\bullet}(p_3)$ is therefore a point infinitely near to $q_1$, which corresponds, on the blow-up $\eta\colon X\to \p^2$, to a point that belong, as proper or infinitely near point, to one of the curves $\mathcal{F}_1$ or $\mathcal{F}_2$, equal respectively to $\mathcal{F}_7$, $\mathcal{F}_6$ (see Figure~\ref{ConfigX2}). Applying $\varphi^{-1}$ to it corresponds to apply $\pi\eta^{-1}$, so $(\varphi^{-k})^{\bullet}(p_3)$ is a point that is infinitely near to $p_6$, and thus to $p_3$. Because $p_3$ is not a base-point of $\varphi^{-i}$ for $1\le i \le k$, the point $p_3$ is not a base-point of $\varphi^{-(k+i)}$ and $(\varphi^{-(k+i)})^{\bullet}(p_3)$ is infinitely near to $(\varphi^{-i})^{\bullet}(p_3)$. In particular, $(\varphi^{-2k})^{\bullet}(p_3)$ is infinitely near to $(\varphi^{-k})^{\bullet}(p_3)$, which is infinitely near to $p_3$. Continuing like this, we get the following assertion:
\begin{equation}\mbox{For any $i\ge 1$, the point $p_3$ is not a base-point of $\varphi^{-i}$.}\label{Assp3}\end{equation}

It remains to show that $p_3$ is a base-point of $\varphi^i$ for each $i\ge 1$ to get the result. To do this, we will need the following  assertion:
\begin{equation}\mbox{For any $i\ge 1$, the point $q_3$ is not a base-point of $\varphi^{i}$.}\label{Assq3}\end{equation}
Note that (\ref{Assq3}) could be proved in the same way as~(\ref{Assp3}), reversing the order of $\varphi$ and $\varphi^{-1}$. We quickly recall the way to deduce it. Note that $q_3$ is not a base-point of $\varphi^i$ for $1\le i\le k-1$, because $q_1$ is not a base-point of $\varphi^i$ (see above). Since $q_3$ does not belong to $B$, which is the set of common base-points of $\varphi^{-1}$ and $\varphi^k$, the point $q_3$ is not a base-point of $\varphi^k$. The point $q_3$ is infinitely near to $q_1=(\varphi^{-(k-1)})^{\bullet}(p_1)$, in the second neighbourhood. The point $(\varphi^{k-1})^{\bullet}(q_3)$ is thus infinitely near to $p_1$. On the blow-up $\pi\colon X\to \p^2$, the point $(\varphi^{k-1})^{\bullet}(q_3)$ corresponds thus to a point that belongs, as a proper or infinitely point, to $\mathcal{E}_1$ or  $\mathcal{E}_2$, equal respectively to  $\mathcal{F}_4$ and $\mathcal{F}_7$. The point $(\varphi^{k})^{\bullet}(q_3)$ is then a point infinitely near to $q_4$, and then to $q_3$. As before, the fact that $q_3$ is not a base-point of $\varphi^i$ for $i=1,\dots,k$ and that $(\varphi^{k})^{\bullet}(q_3)$ is infinitely near to $q_3$ implies that $q_3$ is not a base-point of $\varphi^i$ for any $i\ge 0$, proving Assertion~(\ref{Assq3}).

It remains to see that Assertion~(\ref{Assq3}) implies that $p_3$ is a base-point of $\varphi^i$ for any $i\ge 1$. For $i=1$, this is obvious. For $i>1$, we decompose $\varphi^i$ into $\varphi^{i-1}\circ \varphi$, wedecompose $\pi\colon X\to \p^2$ into $\pi=\pi_{12}\circ \pi_{38}$, where $\pi_{12}\colon Y\to \p^2$ is the blow-up of $p_1,p_2$ and $\pi_{38}\colon X\to Y$ is the blow-up of $p_3,\dots,p_8$, and do the same with $\eta$. This yields the following commutative diagram
$$
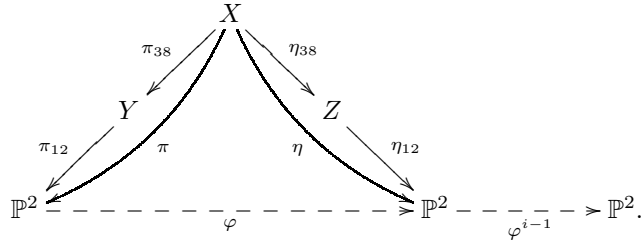
\xymatrix{&& X\ar@/^{-1pc}/[ddrr]_{\eta}\ar@/^{1pc}/[ddll]^{\pi}\ar[ld]_{\pi_{38}}\ar[rd]^{\eta_{38}}\\
& Y\ar[ld]_{\pi_{12}} &&Z\ar[rd]^{\eta_{12}}\\
\p^2\ar@{-->}[rrrr]_{\varphi}&&&&\p^2\ar@{-->}[rr]_{\varphi^{i-1}}&& \p^2.}$$

Note that $\eta_{38}$ contracts $\mathcal{F}_8,\dots,\mathcal{F}_3$ onto the point $q_3\in X_2$, which is not a base-point of $\varphi^{i-1}\circ \eta_{12}$. Let us take the system of conics of $\p^2$ passing through $p_1,p_2,p_3$ and denote by $\Lambda$ the lift of this system on $Y$, which is a system of smooth curves passing through $q_3$ with movable tangents, having dimension $2$. The strict transform on $X$ of $\Lambda$ is a system of curves intersecting $\mathcal{E}_3$ at a general movable point. The map $\eta_{38}$ contracts the curves $\mathcal{C}, \mathcal{E}_2,\mathcal{E}_3,\mathcal{E}_4,\mathcal{E}_1,\mathcal{E}_5$. The curve $\mathcal{E}_3$ being contracted and being not the last one, the image of the system by $\eta_{38}$ passes through $q_3$ with a fixed tangent (corresponding to the point $q_4$). The point $q_3$ being not a base-point of $\varphi^{i-1}\circ \eta_{12}$, the image of the system $\Lambda \subset Y$ by $\varphi^{i-1}\circ \eta\circ (\pi_{38})^{-1}$ has a fixed tangent at the point $(\varphi^{i-1}\circ \eta_{12})(q_3)$. This shows that $p_3$ is a base-point of $\varphi^{i-1}\circ \eta\circ (\pi_{38})^{-1}$, and thus of $\varphi^i$.
\end{proof}

\end{document}